\documentclass[sts]{imsart}

\usepackage{amsmath, amssymb, color}
\usepackage{amsthm} 
\usepackage[colorlinks,citecolor=blue,urlcolor=blue]{hyperref}
\usepackage{xcolor,colortbl}
\usepackage{mathabx}
\usepackage{cleveref, enumitem}
\usepackage{color}
\usepackage{amsfonts, fancybox}
\usepackage{amssymb}
\usepackage[american]{babel}
\usepackage{psfrag,color}
\usepackage{dcolumn}
\usepackage[format=hang, justification=justified, singlelinecheck=false]{subcaption}
\usepackage{flafter, bm, dsfont}
\usepackage[section]{placeins}

\usepackage{multirow}
\usepackage{color}
\usepackage{amsmath}
\usepackage{float}
\usepackage{mathrsfs}
\usepackage{amsfonts}
\usepackage{dsfont}
\usepackage{amssymb}
\usepackage{algorithmic, algorithm}
\usepackage{wrapfig}

\usepackage{amssymb, latexsym}

\usepackage{graphicx}

\numberwithin{equation}{section}

\newtheorem{definition}{Definition}
\newtheorem{theorem}[definition]{Theorem}
\newtheorem{lemma}[definition]{Lemma}

\newtheorem{proposition}[definition]{Proposition}
\newtheorem{conjecture}[definition]{Conjecture}

\newcommand{\cA}{\mathcal{A}}

\newcommand{\cD}{\mathcal{D}}

\newcommand{\cG}{\mathcal{G}}

\newcommand{\cN}{\mathcal{N}}

\newcommand{\cP}{\mathcal{P}}

\newcommand{\cR}{\mathcal{R}}
\newcommand{\cS}{\mathcal{S}}

\newcommand{\cX}{\mathcal{X}}

\renewcommand{\subseteq}{\subset}

\newcommand{\R}{{\rm I}\kern-0.18em{\rm R}}
\newcommand{\h}{{\rm I}\kern-0.18em{\rm H}}
\newcommand{\PP}{{\rm I}\kern-0.18em{\rm P}}
\newcommand{\E}{{\rm I}\kern-0.18em{\rm E}}
\newcommand{\Z}{{\rm Z}\kern-0.18em{\rm Z}}
\newcommand{\1}{{\rm 1}\kern-0.24em{\rm I}}
\newcommand{\N}{{\rm I}\kern-0.18em{\rm N}}

\newcommand{\SN}{\mathcal S_{[N]}}
\newcommand{\SSNs}{\mathcal S_{[N]}^{+}}
\newcommand{\SSN}{\mathcal S_{[N]}^{++}}
\newcommand{\SSNK}{\mathcal S_{[N]}^{(0,1)}}

\newcommand{\ELstar}{\E}
\newcommand{\PLstar}{\PP}

\newcommand{\SX}{\mathcal S_{\cX}}
\newcommand{\SSXs}{\mathcal S_{\cX}^{{\tiny +}}}
\newcommand{\SSX}{\mathcal S_{\cX}^{{\tiny ++}}}
\newcommand{\SSXK}{\mathcal S_{\cX}^{(0,1)}}
\newcommand{\SLambda}{\mathcal S_{\cX}^{\Lambda}}

\newcommand{\DS}{\displaystyle}
\newcommand{\DPP}{\mathsf{DPP}}
\newcommand{\Deg}{\mathsf{Deg}}

\newcommand*\diff{\mathop{}\!\mathrm{d}}
\DeclareMathOperator{\Tr}{Tr}

\DeclareMathOperator{\Diag}{Diag}
\DeclareMathOperator{\diag}{diag}
\DeclareMathOperator{\Var}{Var}
\begin{document}

\begin{frontmatter}

\title{Maximum likelihood estimation of determinantal point processes}
\runtitle{MLE of DPP{s}}

\begin{aug}

\author{\fnms{Victor-Emmanuel}~\snm{Brunel}\ead[label=veb]{vebrunel@math.mit.edu}},
\author{\fnms{Ankur}~\snm{Moitra}\thanksref{t3}\ead[label=ankur]{moitra@mit.edu}},
\author{\fnms{Philippe}~\snm{Rigollet}\thanksref{t2}\ead[label=rigollet]{rigollet@math.mit.edu}}
\and
\author{\fnms{John}~\snm{Urschel}\ead[label=urschel]{urschel@mit.edu}}

\affiliation{Massachusetts Institute of Technology}
\thankstext{t2}{This work was supported in part by NSF CAREER DMS-1541099, NSF DMS-1541100, DARPA BAA-16-46 and a grant from the MIT NEC Corporation.}
\thankstext{t3}{This work was supported in part by NSF CAREER Award CCF-1453261, NSF Large CCF-1565235, a David and Lucile Packard Fellowship, an Alfred P. Sloan Fellowship, an Edmund F. Kelley Research Award, a Google Research Award and a grant from the MIT NEC Corporation.}

\address{{Victor-Emmanuel Brunel}\\
{Department of Mathematics} \\
{Massachusetts Institute of Technology}\\
{77 Massachusetts Avenue,}\\
{Cambridge, MA 02139-4307, USA}\\
\printead{veb}
}

\address{{Ankur Moitra}\\
{Department of Mathematics} \\
{Massachusetts Institute of Technology}\\
{77 Massachusetts Avenue,}\\
{Cambridge, MA 02139-4307, USA}\\
\printead{ankur}
}

\address{{Philippe Rigollet}\\
{Department of Mathematics} \\
{Massachusetts Institute of Technology}\\
{77 Massachusetts Avenue,}\\
{Cambridge, MA 02139-4307, USA}\\
\printead{rigollet}
}

\address{{John Urschel}\\
{Department of Mathematics} \\
{Massachusetts Institute of Technology}\\
{77 Massachusetts Avenue,}\\
{Cambridge, MA 02139-4307, USA}\\
\printead{urschel}
}

\runauthor{Brunel et al.}
\end{aug}

\begin{abstract}
Determinantal point processes (DPPs) have wide-ranging applications in machine learning, where they are used to enforce the notion of diversity in subset selection problems. Many estimators have been proposed, but surprisingly the basic properties of the maximum likelihood estimator (MLE) have received little attention. The difficulty is that it is a non-concave maximization problem, and such functions are notoriously difficult to understand in high dimensions, despite their importance in modern machine learning. Here we study both the local and global geometry of the expected log-likelihood function. We prove several rates of convergence for the MLE and give a complete characterization of the case where these are parametric. We also exhibit a potential curse of dimensionality where the asymptotic variance of the MLE scales exponentially with the dimension of the problem. Moreover, we exhibit an exponential number of saddle points, and give evidence that these may be the only critical points. 
\end{abstract}



\end{frontmatter}

%
%
%
%
%
%
%

\section{Introduction}

Determinantal point processes (DPPs) describe a family of repulsive point processes; they induce probability distributions that favor configurations of points that are far away from each other. DPPs are often split into two categories: discrete and continuous. In the former case, realizations of the DPP are vectors from the Boolean hypercube $\{0,1\}^N$, while in the latter, they occupy a continuous space such as $\R^d$. In both settings, the notion of distance can be understood in the sense of the natural metric with which the space is endowed. Such processes were formally introduced in the context of quantum mechanics to model systems of fermions~\cite{Mac75} that were known to have a repulsive behavior, though DPPs have appeared implicitly in earlier work on random matrix theory, e.g.~\cite{Dys62}.  Since then, they have played a central role in various corners of probability, algebra and combinatorics \cite{BorOls00, BorSos03, Bor11, Oko01, OkoRes03}, for example, by allowing exact computations for integrable systems.

Following the seminal work of Kulesza and Taskar~\cite{KulTas12}, both discrete and continuous DPPs have recently gained attention in the machine learning literature where the repulsive character of DPPs has been used to enforce the notion of diversity in subset selection problems. Such problems are pervasive to a variety of applications such as document or timeline summarization \cite{LinBil12, YaoFanZha16}, image search \cite{KulTas11,AffFoxAda14}, audio signal processing~\cite{XuOu16}, image segmentation~\cite{LeeChaYan16}, bioinformatics \cite{BatQuoKul14}, neuroscience~\cite{SnoZemAda13} and wireless or cellular networks modelization \cite{MiyShi14, TorLeo14, LiBacDhi15, DenZhoHae15}. DPPs have also been employed as methodological tools in Bayesian and spatial statistics~\cite{KojKom16, BisCoe16}, survey sampling~\cite{LooMar15,ChaJolMar16} and Monte Carlo methods~\cite{BarHar16}.

Even though most of the aforementioned applications necessitate estimation of the parameters of a DPP, statistical inference for DPPs has received little attention. In this context, maximum likelihood estimation is a natural method, but generally leads to a non-convex optimization problem. This problem has been addressed by various heuristics, including Expectation-Maximization \cite{GilKulFox14}, MCMC~\cite{AffFoxAda14}, and fixed point algorithms~\cite{MarSra15}. None of these methods come with global guarantees, however. Another route used to overcome the computational issues associated with maximizing the likelihood of DPPs consists in imposing additional modeling constraints, initially in \cite{KulTas12, AffFoxAda14, BarTit15}, and, more recently, ~\cite{DupBac16,GarPaqKoe16,GarPaqKoe16b,MarSra16}, in which assuming a specific low rank structure for the problem enabled the development of sublinear time algorithms.

The statistical properties of the maximum likelihood estimator  for such problems have received attention only in the continuous case and under strong parametric assumptions~\cite{LavMolRub15, BisLav16} or smoothness assumptions in a nonparametric context~\cite{Bar13}. However, despite their acute relevance to machine learning and several algorithmic advances (see~\cite{MarSra15} and references therein), the statistical properties of general discrete DPPs have not been established. Qualitative and quantitative characterizations of the likelihood function would shed light on the convergence rate of the maximum likelihood estimator, as well as aid in the design of new estimators. 

In this paper, we take an information geometric approach to understand the asymptotic properties of the maximum likelihood estimator. First, we study the curvature of the expected log-likelihood around its maximum. Our main result is an exact characterization of when the maximum likelihood estimator converges at a parametric rate (Theorem~\ref{MainCor}). Moreover, we give quantitative bounds on the strong convexity constant (Proposition~\ref{PropositionPath}) that translate into lower bounds on the asymptotic variance and shed light on what combinatorial parameters of a DPP control said variance. Second, we study the global geometry of the expected log-likelihood function. We exhibit an exponential number of saddle points that correspond to partial decouplings of the DPP (Theorem~\ref{ThmCriticalPoints}). We conjecture that these are the only critical points, which would be a key step in showing that the maximum likelihood estimator can be computed efficiently after all, in spite of the fact that it is attempting to maximize a non-concave function. 

The remainder of the paper is as follows. In Section~\ref{SEC:prelim}, we provide an introduction to DPPs together with notions and properties that are useful for our purposes. In Section~\ref{SEC:geometry}, we study the information landscape of DPPs and specifically, the local behavior of the expected log-likelihood around its critical points. Finally, we translate these results into rates of convergence for maximum likelihood estimation in Section~\ref{SEC:stat}. All proofs are gathered in Section~\ref{SEC:proofs} in order to facilitate the narrative.
\subsubsection*{Notation.}
Fix a positive integer $N$ and define $[N]=\{1,2,\ldots,N\}$. Throughout the paper, $\mathcal X$ denotes a subset of  $[N]$. We denote by $\wp(\cX)$ the power set of $\cX$.


We implicitly identify the set of $|\cX| \times |\cX|$ matrices to the the set of mappings from $\cX \times \cX$ to $\R$. As a result,  we denote by $I_\mathcal X$ the identity matrix in $\R^{\mathcal X\times\mathcal X}$ and we omit the subscript whenever $\mathcal X=[N]$. For a matrix $A\in \R^{\mathcal X\times\mathcal X}$ and $J \subset \cX$, denote by $A_J$ the restriction of $A$ to $J \times J$. When defined over $\cX \times \cX$, $A_J$ maps elements outside of $J \times J$ to zero.

Let $\SX$ denote the set of symmetric matrices in $\R^{\mathcal X\times\mathcal X}$ matrices and  denote by $\SLambda$ the subset of matrices in $\SX$ that have eigenvalues in $\Lambda \subset \R$. Of particular interest are $\SSXs=\cS_{\cX}^{[0, \infty)}$, $\SSX=\cS_{\cX}^{(0, \infty)}$, the subsets of positive semidefinite and positive definite matrices respectively.

For a matrix $A \in \R^{\cX \times \cX}$, we denote by $\|A\|_F$, $\det(A)$ and $\Tr(A)$ its Frobenius norm, determinant and trace respectively. We set $\det A_\emptyset=1$ and $\Tr A_\emptyset=0$. Moreover, we denote by $\diag(A)$ the vector of size $|\cX|$ with entries given by the diagonal elements of $A$. If $x \in \R^N$, we denote by $\Diag(x)$ the $N \times N$ diagonal matrix with diagonal given by $x$. 

For $\cA \in \SX$, $k\geq 1$ and a smooth function $f:\mathcal A\to \R$, we denote by $\diff^k f(A)$ the $k$-th derivative of $f$ evaluated at $A\in\mathcal A$. This is a $k$-linear map defined on $\cA$; for $k=1$, $\diff f(A)$ is the gradient of $f$, $\diff^2 f(A)$ the Hessian, etc.

Throughout this paper, we say that a matrix $A \in \SX$ is block diagonal if there exists a partition $\{J_1, \ldots, J_k\}$, $k \ge 1$, of $\cX$ such that $A_{ij}=0$ if $i \in J_a, j \in J_b$ and $a \neq b$. The largest number $k$ such that such a representation exists is called the \emph{number of blocks} of $A$ and in this case $J_1, \ldots, J_k$ are called \emph{blocks} of $L$.

\section{Determinantal point processes and $L$-ensembles}
\label{SEC:prelim}

In this section we gather definitions and useful properties, old and new, about determinantal point processes.
\subsection{Definitions}

A (discrete) \emph{determinantal point process} (DPP) on $\cX$ is a random variable  $Z \in \wp(\cX)$ with distribution 
\begin{equation} \label{DefDPP}
	\PP[J\subseteq Z]=\det(K_J), \hspace{3mm} \forall \,J\subseteq \mathcal X,
\end{equation}
where $K \in \cS_{\cX}^{[0,1]}$, is called the \emph{correlation kernel} of $Z$. 

If it holds further that $K \in \SSXK$, then $Z$ is called \emph{$L$-ensemble}  and there exists a matrix $L=K(I-K)^{-1} \in \SSX$ such that
\begin{equation} \label{DefLEnsemble}
	\PP[Z=J]=\frac{\det(L_J)}{\det(I+L)}, \hspace{3mm} \forall\, J\subseteq \mathcal X,
\end{equation}
Using the multilinearity of the determinant, it is easy to see that   \eqref{DefLEnsemble} defines a probability distribution (see Lemma \ref{keyidentity}). We call $L$ the \textit{kernel} of the $L$-ensemble~$Z$. 

Using the inclusion-exclusion principle, it follows from \eqref{DefDPP} that $\PP(Z=\emptyset)=\det(I-K)$. Hence, a DPP $Z$ with correlation kernel $K$ is an $L$-ensemble if and only if $Z$ can be empty with positive probability.

In this work, we only consider DPPs that are $L$-ensembles. In that setup, we can identify $L$-ensembles and DPPs, and the kernel $L$ and correlation kernel $K$ are related by the identities 
\begin{equation} \label{KtoL}
	L=K(I-K)^{-1}\,,
\qquad 
	K=L(I+L)^{-1}.
\end{equation}

Note that we only consider kernels $L$ that are positive definite. In general $L$-ensembles may also be defined for $L \in \SSXs$, when $K \in \cS_{\cX}^{[0,1)}$. We denote by $\textsf{DPP}_{\mathcal X}(L)$ the probability distribution associated with the DPP with kernel $L$ and  refer to $L$ as the \emph{parameter} of the DPP in the context of statistical estimation. If $\mathcal X=[N]$, we drop the subscript and only write $\textsf{DPP}(L)$ for a DPP with kernel $L$ on $[N]$.

\subsection{Negative association} \label{VertexRep}

Perhaps one of the most distinctive feature of DPPs is their repellent nature. It can be characterized by the notion of \emph{negative association}, which has been extensively covered in the mathematics literature~\cite{BorBraLig09}. To define this notion, we recall that a function $f:\{0,1\}^N \to \R$ is \emph{non decreasing} if for all $x=(x_1, \ldots, x_N)$, $y=(y_1, \ldots, y_N) \in \{0,1\}^N$ such that $x_i \le y_i, \, \forall\, i \in [N]$, it holds that $f(x) \le f(y)$. 

Let $Z$ be a DPP on $[N]$ with kernel $L \in \SSN$ and correlation kernel $K=L(I+L)^{-1} \in \SSNK$. Denote by $\chi(Z) \in \{0,1\}^N$ the (random) characteristic vector of $Z$. Note that $\E[\chi(Z)]=\diag(K)$, moreover, the entries of $\chi(Z)$ are \emph{conditionally negatively associated}.

\begin{definition} \label{NA}
Let $Z$ be a random subset of $[N]$ with characteristic vector $X=\chi(Z) \in \{0,1\}^N$. The coordinates $X_1,\ldots,X_N \in \{0,1\}$ of $X$ are said to be \textit{negatively associated} Bernoulli random variables if for all $J,J' \subset [N]$ such that $J \cap J'=\emptyset$ and  all non decreasing functions $f$ and $g$ on $\{0,1\}^N$, it holds
$$
\E\big[f(\chi(Z \cap J))g(\chi(Z \cap J'))\big]\leq \E\big[f(\chi(Z \cap J))\big]\E\big[g(\chi(Z \cap J'))\big]\,.
$$
Moreover, $X_1,\ldots,X_N$ are \textit{conditionally negatively associated} if it also holds that for all $S\subseteq [N]$\,,
\begin{align*}
& \E\big[f(\chi(Z \cap J))g(\chi(Z \cap J'))\big| Z \cap S\big] \\
& \hspace{20mm} \leq \E\big[f(\chi(Z \cap J))\big| Z \cap S\big]\E\big[g(\chi(Z \cap J'))\big| Z \cap S\big]\,
\end{align*}
almost surely.
\end{definition}

Negative association is much stronger than pairwise non positive correlations. Conditional negative association is even stronger, and this property will be essential for the proof of Theorem \ref{ThmCriticalPoints}. The following lemma is a direct consequence of Theorem 3.4 of \cite{BorBraLig09}.

\begin{lemma}
\label{LEM:CNA}
Let $Z \sim \DPP(L)$ for some $L \in \SSN$ and denote by $\chi(Z)=(X_1, \ldots, X_N) \in \{0,1\}^N$ its characteristic vector. Then, the	Bernoulli random variables $X_1,\ldots,X_N$ are conditionally negatively associated.
\end{lemma}

Now we introduce the notion of a \textit{partial decoupling} of a DPP. This notion will be relevant in the study of the likelihood geometry of DPPs. 

\begin{definition}
Let $\cP$ be a partition of $[N]$. A \emph{partial decoupling} $Z'$ of a DPP $Z$ on $[N]$ according to partition~$\cP$ is a random subset of $[N]$ such that $\{\chi(Z' \cap J), J \in \cP\}$ are mutually independent and $\chi(Z' \cap J)$ has the same distribution as $\chi(Z \cap J)$ for all $J \in \cP$. We say that the partial decoupling is \textit{strict} if and only if $Z'$ does not have the same distribution as $Z$.
\end{definition}

It is not hard to see that a partial decoupling $Z'$ associated to a partition $\cP$ of a DPP $Z$ is also a DPP with correlation kernel $K'$ given by	
	$$K'_{i,j} = \left\{
	\begin{array}{ll}
		K_{i,j} & \text{if} \ i,j\in J\  \text{for some}\  J\in \cP\,,\\
		0 & \text{otherwise.}
\end{array}
\right.
	$$	
	
In particular, note that if $Y'$ is a strict partial decoupling of a DPP $Y$, then its correlation kernel $K$ and thus its kernel $L$ are both block diagonal with at least two blocks.

\subsection{Identifiability}

The probability mass function~\eqref{DefLEnsemble} of $\textsf{DPP}(L)$ depends only on the principal minors of $L$ and on $\det(I+L)$. In particular, $L$ is not fully identified by $\textsf{DPP}(L)$ and the lack of identifiability of $L$ has been characterized exactly \cite[Theorem 4.1]{Kul12}. Denote by $\mathcal D$ the collection of $N\times N$ diagonal matrices with $\pm 1$ diagonal entries. Then, for $L_1,L_2\in\SSN$, 
\begin{equation} \label{Identifiability} 
	\textsf{DPP}(L_1)=\textsf{DPP}(L_2) \iff \exists D\in\mathcal D, L_2=DL_1D.
\end{equation}
We define the degree of identifiability of a kernel $L$ as follows.
\begin{definition} 
\label{DefReducible}
Let $L\in\SSN$. The degree $\Deg(L)$ of identifiability of $L$ is the cardinality of the family $\{DLD: D\in\mathcal D\}$. We say that $L$ is \emph{irreducible} whenever $\Deg(L)=2^{N-1}$ and \emph{reducible} otherwise. If $Z \sim \DPP(L)$ for some $L \in \SSN$, we also say that $Z$ is irreducible if $L$ is irreducible, and that $Z$ is reducible otherwise.
\end{definition}

For instance, the degree of identifiability of a diagonal kernel is 1. It is easy to check that diagonal kernels are the only ones with degree of identifiability equal to 1. These kernels are perfectly identified. Intuitively, the higher the degree, the less the kernel is identified. It is clear that for all $L\in\SSN, 1\leq \textsf{Deg}(L)\leq 2^{N-1}$. 

As we will see in Proposition \ref{Prop1}, the degree of identifiability of a kernel $L$ is completely determined by its block structure.
The latter can in turn be characterized by the connectivity of certain graphs that we call \emph{determinantal graphs}.

\begin{definition}
Fix $\cX \subset [N]$. The \emph{determinantal graph} $\mathcal G_L=(\cX, E_L)$ of a DPP with kernel $L \in \SSX$ is the undirected graph with vertices $\cX$ and edge set $E_L=\big\{\{i,j\}\,:\, L_{i,j}\neq 0\big\}$. If $i,j\in \mathcal X$, write $i\sim_L j$ if there exists a path in $\mathcal G_L$ that connects $i$ and $j$. 
\end{definition}

It is not hard to see that a DPP with kernel $L$ is irreducible if and only if its determinantal graph $\cG_L$ is connected. The blocks of $L$ correspond to the connected components of $\cG_L$. Moreover, it follows directly from~(\ref{DefLEnsemble}) that if $Z\sim \DPP(L)$ and $L$ has blocks $J_1, \ldots, J_k$, then $Z\cap J_1, \ldots, Z\cap J_k$ are mutually independent DPPs with correlation kernels $K_{J_1}, \ldots, K_{J_k}$ respectively, where $K=L(I+L)^{-1}$ is the correlation kernel of $Z$.

The main properties regarding identifiability of DPPs are gathered in the following straightforward proposition.

\begin{proposition} \label{Prop1}
Let $L\in\SSN$ and $Z \sim \DPP(L)$. Let $1\leq k\leq N$ and $\{J_1,\ldots,J_k\}$ be a partition of $[N]$. The following statements are equivalent:
\begin{enumerate}
	\item $L$ is block diagonal with $k$ blocks $J_1, \ldots, J_k$,
	\item $K$ is block diagonal with $k$ blocks $J_1, \ldots, J_k$,
	\item $Z\cap J_1, \ldots, Z\cap J_k$ are mutually independent irreducible DPPs,
	\item $\mathcal G_L$ has $k$ connected components given by $J_1, \ldots, J_k$,
	\item $L=D_jLD_j$, for $D_j=\Diag(2\chi(J_j)-1) \in \cD$, for all $j \in [k]$.
\end{enumerate}
\end{proposition}

In particular, Proposition \ref{Prop1} shows that the degree of identifiability of $L\in\SSN$ is $\Deg(L)=2^{N-k}$, where $k$ is the number of blocks of $L$.

Now that we have reviewed useful properties of DPPs, we are in a position to study the information landscape for the statistical problem of estimating the kernel of a DPP from independent observations.

%
%
%
%
%

%
%

\section{Geometry of the likelihood functions}
\label{SEC:geometry}
\subsection{Definitions}

Our goal is to estimate an unknown kernel $L^* \in \SSN$ from $n$ independent copies of $Z \sim \DPP(L^*)$. In this paper, we study the statistical properties of what is arguably the most natural estimation technique: maximum likelihood estimation.

%


Let $Z_1,\ldots, Z_n$ be $n$ independent copies of $Z\sim \DPP(L^*)$ for some unknown $L^*\in\SSN$.  The (scaled) log-likelihood associated to this model is given for any $L\in\SSN$,
\begin{equation}
\label{EmpLogLike}
\hat\Phi(L)  = \frac{1}{n}\sum_{i=1}^n \log p_{Z_i}(L)  = \sum_{J\subseteq [N]} \hat p_J\log\det(L_J) - \log\det(I+L)\,,
\end{equation}
where $p_J(L)=\PP[Z=J]$ is defined in~\eqref{DefLEnsemble} and $\hat p_J$ is its empirical counterpart defined by
$$
\hat p_J=\frac{1}{n} \sum_{i=1}^n \1(Z_i=J)\,.
$$
Here $\1(\cdot)$ denotes the indicator function.

Using the identity \eqref{KtoL}, it is also possible to write $p_J(L)$ as 
$$p_J(L)=|\det(K-I_{\bar J})|,$$
where $\bar J$ is the complement of $J$. Hence, the log-likelihood function can be defined with respect to $K\in\SSNK$ as
\begin{equation} \label{EmpLogLikeK}
	\hat\Psi(K) = \sum_{J\subseteq [N]} \hat p_J\log|\det(K-I_{\bar J})|\,.
\end{equation}
We denote by $\Phi_{L^*}$ (resp. $\Psi_{L^*}$) the expected log-likelihood as a function of $L$ (resp. $K$):
\begin{equation} \label{ExpLogLike}
	\Phi_{L^*}(L) = \sum_{J\subseteq [N]} p_J(L^*)\log\det(L_J) - \log\det(I+L)\,.
	\end{equation}
and
\begin{equation} \label{ExpLogLikeK}
	\Psi_{L^*}(K) = \sum_{J\subseteq [N]} p_J(L^*)\log|\det(K-I_{\bar J})|\,.
\end{equation}

For the ease of notation, we assume in the sequel that $L^*$ is fixed, and write simply $\Phi=\Phi_{L^*}$, $\Psi=\Psi_{L^*}$ and $p_J^*=p_J(L^*)$, for $J\subseteq [N]$.

We now proceed to studying the function $\Phi$. Namely, we study its critical points and their type: local/global maxima, minima and saddle points. We also give a necessary and sufficient condition on $L^*$ so that $\Phi$ is locally strongly concave around $L=L^*$, i.e., the Hessian of $\Phi$ evaluated at $L=L^*$ is definite negative. 
All our results can also be rephrased in terms of $\Psi$.

\subsection{Global maxima}

Note that $\Phi(L)$ is, up to an additive constant that does not depend on $L$, the Kullback-Leibler (KL) divergence between $\DPP(L)$ and $\DPP(L^*)$: 

$$
	\Phi(L) = \Phi(L^*) - \textsf{KL}\left(\textsf{DPP}(L^*),\textsf{DPP}(L)\right), \forall L\in\SSN\,,
$$
where $\textsf{KL}$ stands for the Kullback-Leibler divergence between probability measures. In particular, by the properties of this divergence, $\Phi(L)\leq \Phi(L^*)$ for all $L\in\SSN$, and
$$
	\Phi(L)=\Phi(L^*) 
	\iff \textsf{DPP}(L)=\textsf{DPP}(L^*) 
	 \iff L=DL^*D, \hspace{3mm} \mbox{for some } D\in\mathcal D.
$$
As a consequence, the global maxima of $\Phi$ are exactly the matrices $DL^*D$, for $D$ ranging in $\mathcal D$. The following theorem gives a more precise description of $\Phi$ around $L^*$ (and, equivalently, around each $DL^*D$ for $D\in\mathcal D$).

\begin{theorem} \label{MainThm}

Let $L^*\in\SSN$, $Z\sim\textsf{DPP}(L^*)$ and $\Phi=\Phi_{L^*}$, as defined in \eqref{ExpLogLike}. Then, $L^*$ is a critical point of $\Phi$. Moreover, for any $H \in \SN$, 
\begin{equation*}
\diff^2\Phi(L^*)(H,H)=-\Var[ \Tr((L_Z^*)^{-1} H_Z)].
\end{equation*} 

In particular, the Hessian $\diff^2\Phi(L^*)$ is negative semidefinite.
\end{theorem}

The first part of this theorem is a consequence of the facts that $L^*$ is a global maximum of a smooth $\Phi$ over the open parameter space $\SSN$. The second part of this theorem follows from the usual fact that the Fisher information matrix has two expressions: the opposite of the Hessian of the expected log-likelihood and the variance of the score (derivative of the expected log-likelihood). We also provide a purely algebraic proof of \ref{MainThm} in the appendix.

Our next result characterizes the null space of $d^2\Phi(L^*)$ in terms of the determinantal graph $\mathcal G_{L^*}$.

\begin{theorem} \label{MainCor}

Under the same assumptions of Theorem \ref{MainThm}, the null space of the quadratic Hessian map $H\in\SN \mapsto \diff^2\Phi(L^*)(H,H)$ is given by 
\begin{equation}
\label{EQ:defNL}
\mathcal N(L^*)=\left\{H\in\SN\,:\,H_{i,j}=0\ \text{for all } i,j \in [N] \ \text{such that}\ \ i\sim_{L^*}j\right\}\,.
\end{equation}
In particular, $\diff^2\Phi(L^*)$ is definite negative if and only if $L^*$ is irreducible.
\end{theorem}

The set $\mathcal N(L^*)$ has an interesting interpretation using perturbation analysis when $L^*$ is reducible. On the one hand, since $L^*$ is reducible, there exits $D_0\in\mathcal D \setminus \{-I, I\}$ such that $L^*=D_0L^*D_0$ is a global maximum for $\Phi_{L^*}$. On the other hand, take any $H\in\mathcal \SN$  such that $L^*+H\in \SSN$ and observe that $D(L^*+H)D$ are all global maxima for $\Phi_{L^*+H}$ and in particular,  $D_0(L^*+H)D_0$ is a global maximum for $\Phi_{L^*+H}$. The Frobenius distance between $L^*$ and $D_0(L^*+H)D_0$ is $\|H-D_0HD_0\|_F$, which is maximized over $H$ with fixed norm if and only if $D_0HD_0=-H$. Such matrices span precisely the null space $\mathcal N(L^*)$ (see Lemma \ref{SpanNullSpace}). This leads to the following interpretation of $\mathcal N(L^*)$: The directions along which $\Phi_{L^*}$ has vanishing second derivative $L=L^*$ are spanned by the matrices $H$ that push away any two merged modes of $\Phi_L^*$ as much as possible.
%
%
%
%
%
%
%
%

It follows from Theorem~\ref{MainCor} that  $\Phi_{L^*}$ is locally strongly concave around $L^*$ if and only if $L^*$ is irreducible since, in that case, the smallest eigenvalue of $-\diff^2\Phi(L^*)$ is positive. Nevertheless, this positive eigenvalue may be exponentially small in $N$, leading to a small curvature around the maximum of $\Phi_{L^*}$. This phenomenon is illustrated by the following example.

Consider the tridiagonal matrix $L^*$ given by:
$$L_{i,j}^* = \begin{cases}
		a \mbox{ if } i=j,\\
		b \mbox{ if } |i-j|=1, \\
		0 \mbox{ otherwise,}
\end{cases}
	$$	
where $a$ and $b$ are real numbers. 

\begin{proposition} \label{PropositionPath}
Assume that $a>0$ and $a^2>2b^2$. Then, $L^*\in\SSN$ and there exist two positive numbers $c_1$ and $c_2$ that depend only on $a$ and $b$ such that
\begin{equation*}
	0 <  \inf_{H\in\SN:\|H\|_F=1}-\diff^2\Phi(L^*)(H,H)\leq c_1e^{-c_2N}.
\end{equation*}
\end{proposition}

%
%

While the Hessian cancels in some directions $H \in \cN(L^*)$ for any reducible $L^* \in \SSN$, the next theorem shows that the fourth derivative is negative in \emph{any} nonzero direction $H \in \cN(L^*)$ so that $\Phi$ is actually curved around $L^*$ in any direction.


\begin{theorem} \label{FourthOrder}

Let $H\in\mathcal N(L^*)$. Then,
\begin{itemize}
	\item[(i)] $\displaystyle{\diff^3\Phi(L^*)(H,H,H)=0}$;
	\item[(ii)] $\displaystyle{\diff^4\Phi(L^*)(H,H,H,H)=-\frac{2}{3}\Var\left[\Tr\left(((L_Z^*)^{-1}H_Z)^2\right)\right]}\leq 0$;
	\item[(iii)] $\displaystyle{\diff^4\Phi(L^*)(H,H,H,H)=0 \iff H=0}$.
\end{itemize}

\end{theorem}

The first part of Theorem \ref{FourthOrder} is obvious, since $L^*$ is a global maximum of $\Phi$. However, we give an algebraic proof of this fact, which is instructive for the proof of the two remaining parts of the theorem.





\subsection{Other critical points}

The function $\Phi_{L^*}$ is not concave and so finding its global maximum is fraught with difficulty. A standard approach is to work with a concave relaxation \cite{CanTao04, CanRec09, AbbBanHal16}, which has proven to be successful in applications such as compressed sensing, matrix completion and community detection. More recently, algorithms that attempt to directly optimize a non-concave objective have received growing attention, 
primarily driven by a good empirical performance and simple implementation (see \cite{AroGeMa15, CanLiSol15, BalWaiYu17} for example). 

In fact, there are {\em two} issues that confound such approaches. The first is spurious local maxima where gradient ascent can get trapped. In some instances such as matrix completion~\cite{GeLeeMa16} it can be shown that the non-concave objective has no spurious local maxima, while in others such as Gaussian mixture models~\cite{JinZhaBal16}, it does. The second issue is the presence of a large and often exponential number of saddle points. Empirically, it has been postulated~\cite{Dauphin:2014} that escaping saddle points is the main difficulty in optimizing large non-concave objectives. However if certain conditions on the saddle points are met then it is known that one can efficiently find a local maximum \cite{Nesterov2006, GeHJY15}. 

Here we show that the function $\Phi_{L^*}$ has exponentially many saddle points that correspond to all possible partial decouplings of the DPP.

\begin{theorem} \label{ThmCriticalPoints}
Let $L^*\in\SSN$ and $K^*=L^*(I+L^*)^{-1}$. Let $Z\sim\textsf{DPP}(L^*)$. Then, the kernel $L$ of any partial decoupling of $Z$ is a critical point of $\Phi_{L^*}$. Moreover, it is always a saddle point when the partial decoupling is strict.
\end{theorem}

We conjecture that these are the only saddle points, which would be a major step in showing that despite the fact that  $\Phi_{L^*}$ is non-concave, one can find its maximum via first and second order methods. This would give a compelling new example of a  problem arising from big data where non-concave optimization problems can be tamed. 

\begin{conjecture}\label{ConjCriticalPoints}
	Let $L^*\in\SSN$ and $Z\sim\textsf{DPP}(L^*)$. The kernels of the partial decouplings of $Z$ are the only critical points of $\Phi_{L^*}$.
\end{conjecture}
%

The following proposition provides some evidence, by verifying a consequence of the conjecture:

\begin{proposition} \label{PropConjecture}
	Let $L^*\in\SSN$ and let $L$ be a critical point of $\Phi_{L^*}$. Let $K^*=L^*(I+L^*)^{-1}$ and $K=L(I+L)^{-1}$. Then, $K^*$ and $K$ have the same diagonal.
\end{proposition}

\section{Maximum likelihood estimation}
\label{SEC:stat}
Let $Z_1,\ldots,Z_n$ be $n$ independent copies of $Z \sim \DPP(L^*)$ with unknown kernel $L^*\in\SSN$. The maximum likelihood estimator (\textit{MLE}) $\hat L$ of $L^*$ is defined as a maximizer of the likelihood $\hat \Phi$ defined in \eqref{EmpLogLike}. Since for all $L\in\SSN$ and all $D\in\mathcal D$, $\hat\Phi(L)=\hat\Phi(DLD)$, there is more than one kernel $\hat L$ that maximizes $\hat\Phi$ in general. We will abuse notation and refer to any such maximizer as \emph{``the" MLE}. Since there is a bijection \eqref{KtoL} between kernels $L$ and correlation kernels $K$, the random correlation kernel $\hat K=\hat L(I+\hat L)^{-1}$ maximizes the function $\hat\Psi$ defined in \eqref{EmpLogLikeK} and therefore, is the maximum likelihood estimator of the unknown correlation kernel $K^*=L^*(I+L^*)^{-1}$. 


We measure the performance of the MLE using the \emph{loss} $\ell$ defined by
$$
	\ell(\hat L,L^*)=\min_{D\in\mathcal D}\|\hat L-DL^*D\|_F\,
$$
where we recall that $\|\cdot\|_F$ denotes the Frobenius norm.

The loss $\ell(\hat L, L^*)$ being a random quantity, we also define its associated \emph{risk}~$\cR_n$ by
$$
\mathcal R_n(\hat L,L^*)=\E\big[\ell(\hat L,L^*)\big],
$$
where the expectation is taken with respect to the joint distribution of the iid observation $Z_1,\ldots,Z_n \sim \DPP(L^*)$. 

Our first statistical result establishes that the MLE is a consistent estimator.
\begin{theorem} \label{Consistency}
	$$\ell(\hat L,L^*)\xrightarrow[n\to\infty]{} 0\,,\qquad \text{in probability.}$$
\end{theorem}


Theorem \ref{Consistency} shows that consistency of the MLE holds for all $L^*\in\SSN$. However, the MLE can be $\sqrt n$-consistent only when $L^*$ is irreducible. Indeed, this is the only case when the Fisher information is invertible, by Theorem \ref{MainCor}.

Let $M\in\SN$ and $\Sigma$ be a symmetric, positive definite bilinear form on $\SN$. We write $A\sim\mathcal N_{\SN}(M,\Sigma)$ to denote a Wigner random matrix $A\in \SN$, such that for all $H\in\SN$, $\Tr(AH)$ is a Gaussian random variable, with mean $\Tr(M H)$ and variance $\Sigma(H,H)$.

Assume that $L^*$ is irreducible and let $\hat L$ be the MLE. Let $\hat D\in\mathcal D$ be such that 
$$\|\hat D\hat L\hat D-L^*\|_F=\min_{D\in\mathcal D}\|D\hat L D-L^*\|_F$$ 
and set $\tilde L=\hat D\hat L\hat D$. Recall that by Theorem~\ref{MainCor}, the bilinear operator $\diff^2\Phi(L^*)$ is invertible and let $V(L^*)$ be denote its inverse. Then, by Theorem 5.41 in \cite{Vaa98},
\begin{equation}
	\sqrt n (\tilde L-L^*)=-V(L^*)\frac{1}{\sqrt n}\sum_{i=1}^n\left((L_{Z_i}^*)^{-1}-(I+L^*)^{-1}\right)+\rho_n,
\end{equation}
where $\DS \|\rho_n\|_F\xrightarrow[n\to\infty]{}  0.$
Hence, we get the following theorem.

\begin{theorem} \label{AsymNormMLE}
	Let $L^*$ be irreducible. Then, $\tilde L$ is asymptotically normal, with asymptotic covariance operator $V(L^*)$:
	$$\sqrt n (\tilde L-L^*)\xrightarrow[n\to\infty]{} \mathcal N_{\SN}\left(0,V(L^*)\right)\,,
	$$
	where the above convergence holds in distribution.
\end{theorem}

Recall that we exhibited in Proposition~\ref{PropositionPath} an irreducible kernel $L^* \in \SSN$ that is non-degenerate---its entries and eigenvalues are either zero or bounded away from zero---such that $V(L^*)[H,H]\ge c^N$ for some positive constant $c$ and unit norm $H \in \SN$. Together with Theorem~\ref{AsymNormMLE}, it implies that while the MLE $\tilde L$ converges at the parametric rate $n^{1/2}$,  $\sqrt{n}\Tr[{(\tilde L-L^*)^\top H}]$ has asymptotic variance of order at least $c^N$ for some positive constant $c$. It implies that the MLE suffers from a \emph{curse of dimensionality}.

In the sequel, we say that an estimator $\hat\theta$ of an unknown quantity $\theta$ is $n^\alpha$-consistent (for a given $\alpha>0$) if the sequence $n^{\alpha}(\hat\theta-\theta)$ is bounded in probability. In particular, if the sequence $n^{\alpha}(\hat\theta-\theta)$ converges in distribution, then $\hat\theta$ is $n^\alpha$-consistent.

When $L^*$ is not irreducible, the MLE is no longer a $\sqrt{n}$-consistent estimator of $L^*$; it is only $n^{1/6}$-consistent. Nevertheless, in this case, the blocks of $L^*$ may still be estimated at the parametric rate, as indicated by the following theorem.

 If $A\in\R^{N\times N}$ and $J,J'\subseteq [N]$, we denote by $A_{J,J'}$ the $N\times N$ matrix whose entry $(i,j)$ is $A_{i,j}$ if $(i,j)\in J\times J'$ and $0$ otherwise. We have the following theorem.

\begin{theorem} \label{AsympNormMLE}
	Let $L^*\in\SSN$ be block diagonal with blocks $\cP$. Then, for $J, J' \in \cP$, $J \neq J'$,
\begin{equation}
	\min_{D\in\mathcal D}\|\hat L_{J,J'}-DL_{J,J'}^*D\|_F=O_{\PLstar}(n^{-1/6})
\end{equation}
and 
\begin{equation}
	\min_{D\in\mathcal D}\|\hat L_{J}-DL_{J}^*D\|_F=O_{\PLstar}(n^{-1/2}).
\end{equation}
\end{theorem}


Theorem \ref{AsympNormMLE} may also be stated in terms of $K^*$ and its MLE $\hat K=\hat L(I+\hat L)^{-1}$. In particular, the MLE $\hat K$ estimates the diagonal entries of $K^*$ at the speed $n^{-1/2}$, no matter whether $L^*$ (or, equivalently, $K^*$) is irreducible. Actually, it is possible to compute $\hat K_{j,j}$, for all $j\in [N]$: It is equal to the estimator of $K_{j,j}^*$ obtained by the method of moments. Indeed, recall that $\hat L$ satisfies the first order condition
$$\sum_{J\subseteq [N]} \hat p_J \hat L_J^{-1}=(I+\hat L)^{-1}.$$
Post-multiplying by $\hat L$ both sides of this equality and identifying the diagonal entries yields
\begin{align*}
	\hat K_{j,j} & = \sum_{J\subseteq [N]:J\ni j}\hat p_J  = \frac{1}{n}\sum_{i=1}^n\mathds 1_{j\in Z_i},
\end{align*}
for all $j=1,\ldots,N$. This is the estimator of $K_{j,j}^*$ obtained by the method of moments and it is $\sqrt n$-consistent by the central limit theorem.

\section{Conclusion and open problems}


In this paper, we studied the local and global geometry of the log-likelihood function. We gave a nuanced treatment of the rates achievable by the maximum likelihood estimator and we establish when it can achieve parametric rates, and even when it cannot, which sets of parameters are the bottleneck. The main open question is to resolve Conjecture~\ref{ConjCriticalPoints}, which would complete our geometric picture of the log-likelihood function. 

In a companion paper~\cite{BruMoiRig17}, using an approach based on the method of moments, we devise an efficient method to compute an estimator that converges at a parametric rate for a large family of kernels. Moreover, the running time and sample complexity are polynomial in the dimension of the DPP, even though here we have shown that the strong convexity constant can be exponentially small in the dimension.

\section{Proofs}
\label{SEC:proofs}
\subsection{A key determinantal identity and its consequences}

We start this section by giving a key yet simple identity for determinants.

\begin{lemma} \label{keyidentity}
For all square matrices $L\in\R^{N\times N}$,
\begin{equation} \label{keyEq}
	\det(I+L)=\sum_{J\subseteq [N]}\det(L_J).
\end{equation}
\end{lemma}
This identity is a direct consequence of the multilinearity of the determinant. Note that it gives the value of the normalizing constant in \eqref{DefLEnsemble}. Successive differentiations of~\eqref{keyEq} with respect to $L$ lead to further useful identities. To that end, recall that if $f(L)=\log\det(L), L\in\SSN$, then for all $H\in\SN$,
$$\diff f(L)(H)=\Tr(L^{-1}H).$$
Differentiating \eqref{keyEq} once over $L\in\SSN$ yields

\begin{equation} \label{keyEq1_0}
	\sum_{J\subseteq [N]}\det(L_J)\Tr(L_J^{-1}H_J)=\det(I+L)\Tr((I+L)^{-1}H), \hspace{3mm} \forall H\in\SN.
\end{equation}

In particular, after dividing by $\det(I+L)$,
\begin{equation} \label{keyEq1}
	\sum_{J\subseteq [N]}p_J(L)\Tr(L_J^{-1}H_J)=\Tr((I+L)^{-1}H), \hspace{3mm} \forall H\in\SN.
\end{equation}

In matrix form, \eqref{keyEq1} becomes
\begin{equation} \label{keyEq1MatrixForm}
	\sum_{J\subseteq [N]}p_J(L)L_J^{-1}=(I+L)^{-1}.
\end{equation}
Here we use a slight abuse of notation. For $J\subseteq [N]$, $L_J^{-1}$ (the inverse of $L_J$) has size $|J|$, but we still denote by $L_J^{-1}$ the $N\times N$ matrix whose restriction to $J$ is $L_J^{-1}$ and which has zeros everywhere else. 

Let us introduce some extra notation, for the sake of presentation. For any positive integer $k$ and $J\subseteq [N]$, define 
$$
 a_{J,k}=\Tr\big((L_J^{-1}H_J)^k\big) \quad \text{and} \quad a_k=\Tr\big(((I+L)^{-1}H)^k\big)\,,
 $$ 
 where we omit the dependency in $H\in\SN$. Then, differentiating again \eqref{keyEq1_0} and rearranging terms yields
\begin{equation} \label{keyEq2}
	\sum_{J\subseteq [N]}p_J(L)a_{J,2}-a_2=\sum_{J\subseteq [N]}p_J(L)a_{J,1}^2-a_1^2,
\end{equation}
for all $H\in\SN.$
In the same fashion, further differentiations yield
\begin{align}
	\sum_{J\subseteq [N]}p_J(L)a_{J,3}-a_3 & = -\frac{1}{3}\Big(\sum_{J\subseteq [N]}p_J(L)a_{J,1}^3-a_1^3\Big)+\frac{2}{3}\Big(\sum_{J\subseteq [N]}p_J(L)a_{J,2}-a_2\Big) \nonumber \\
	 \label{keyEq3} & \hspace{8mm}  +\frac{1}{3}\Big(\sum_{J\subseteq [N]}p_J(L)a_{J,1}a_{J,2}-a_1a_2\Big)
\end{align}
and 
\begin{align} 
	\sum_{J\subseteq [N]}&p_J(L)a_{J,4}-a_4 \nonumber \\
	&= \frac{1}{9}\Big(\sum_{J\subseteq [N]}p_J(L)a_{J,1}^4-a_1^4\Big)-\frac{4}{9}\Big(\sum_{J\subseteq [N]}p_J(L)a_{J,1}^2a_{J,2}-a_1^2 a_2\Big) \nonumber \\
	& -\frac{2}{9}\Big(\sum_{J\subseteq [N]}p_J(L)a_{J,1}a_{J,2}-a_1a_2\Big)+\frac{5}{9}\Big(\sum_{J\subseteq [N]}p_J(L)a_{J,1}a_{J,3}-a_1a_3\Big) \nonumber \\
	\label{keyEq4} & +\frac{1}{9}\Big(\sum_{J\subseteq [N]}p_J(L)a_{J,2}^2-a_2^2\Big) +\frac{4}{9}\Big(\sum_{J\subseteq [N]}p_J(L)a_{J,3}-a_3\Big),
\end{align}
for all $H\in\SN.$

\subsection{The derivatives of $\Phi$}

Let $L^*\in\SSN$ and $\Phi=\Phi_{L^*}$. In this section, we give the general formula for the derivatives of $\Phi$.

\begin{lemma} \label{Derivatives}
	For all positive integers $k$ and all $H\in\SN$,
\begin{align*}
	& \diff^k\Phi(L^*)(H, \ldots, H) \\
	& \hspace{8mm} = (-1)^{k-1}(k-1)!\left(\sum_{J\subseteq [N]} p_J^*\Tr\left(((L_J^*)^{-1}H_J)^k\right)-\Tr\left(((I+L^*)^{-1}H)^k\right)\right).
\end{align*}
\end{lemma}

\paragraph{Proof\\} This lemma can be proven by induction, using the two following facts. If $f(M)=\log\det(M)$ and $g(M)=M^{-1}$ for $M\in\SSN$, then for all $M\in\SSN$ and $H\in\SN$,
$$\diff f(M)(H)=\Tr(M^{-1}H)$$
and
$$\diff g(M)(H)=-M^{-1}HM^{-1}.$$
\hfill \textsquare

\subsection{Auxiliary lemma}

\begin{lemma} \label{SpanNullSpace}
Let $L^*\in\SSN$ and $\mathcal N(L^*)$ be defined as in~\eqref{EQ:defNL}. Let $H\in\mathcal N(L^*)$. Then, $H$ can be decomposed as  $H=H^{(1)}+\ldots +H^{(k)}$ where for each $j=1,\ldots,k$,  $H^{(j)}\in\SN$ is such that $D^{(j)}H^{(j)}D^{(j)}=-H^{(j)}$, for some $D^{(j)}\in\mathcal D$ satisfying $D^{(j)}L^*D^{(j)}=L^*$.
\end{lemma}
\begin{proof}
Let $H\in\mathcal N(L^*)$. Denote by $J_1,\ldots,J_M$ the blocks of $L^*$ ($M=1$ and $J_1=[N]$ whenever $L^*$ is irreducible). For $i=1,\ldots,M$, let $D^{(i)}=\Diag(2\chi(J_i)-1)\in\mathcal D$. Hence, $D^{(i)}L^*D^{(i)}=L^*$, for all $i=1,\ldots,k$.

For $i,j\in [k]$ with $i< j$, define
$$
H^{(i,j)}=\Diag(\chi(J_i))H\Diag(\chi(J_j))+\Diag(\chi(J_j))H\Diag(\chi(J_i))\,.
$$
%
Then, it is clear that
$$
H=\sum_{1\leq i<j\leq M}H^{(i,j)} \qquad \text{and} \qquad
	D^{(i)}H^{(i,j)}D^{(i)}=-H^{(i,j)},\ \forall\, i<j\,.
$$
The lemma follows by renumbering the matrices $H^{(i,j)}$. 
\end{proof}

\subsection{Proof of Theorem \ref{MainThm}}


Theorem \ref{MainThm} is a direct consequence of Lemma \ref{Derivatives} and identities \eqref{keyEq1} and \eqref{keyEq2}. 

\hfill \textsquare


\subsection{Proof of Theorem \ref{MainCor}}

Let $H\in\SN$ be in the null space of $\diff^2\Phi(L^*)$, i.e., satisfy $\diff^2\Phi(L^*)(H,H)=0$. We need to prove that $H_{i,j}=0$ for all pairs $i,j\in [N]$ such that $i\sim_{L^*}j$. To that end, we proceed by (strong) induction on the distance between $i$ and $j$ in $\mathcal G_{L^*}$, i.e., the length of the shortest path from $i$ to $j$ (equal to $\infty$ if there is no such path). Denote this distance by $d(i,j)$.

First, by Theorem \ref{MainThm}, $\Var[ \Tr((L_Z^*)^{-1} H_Z)]=0$ so the random variable $\Tr((L_Z^*)^{-1} H_Z)$ takes only one value with probability one. Therefore since $p_J^*>0$ for all $J \subset [N]$ and $\Tr((L_\emptyset^*)^{-1} H_\emptyset) = 0$,   we also have
\begin{equation} \label{ConstTrace}
\Tr([L^*_J]^{-1} H_J) = 0, \quad \forall J \subset [N].
\end{equation}

We now proceed to the induction.

If $d(i,j)=0$, then $i=j$ and since $L^*$ is definite positive, $L_{i,i}^*\neq 0$. Thus, using \eqref{ConstTrace} with $J=\{i\}$, we get $H_{i,i}=0$. 

If $d(i,j)=1$, then $L_{i,j}^*\neq 0$, yielding $H_{i,j}=0$, using again \eqref{ConstTrace}, with $J=\{i,j\}$ and the fact that $H_{i,i}=H_{j,j}=0$, established above.

Let now $m\geq 2$ be an integer and assume that for all pairs $(i,j)\in [N]^2$ satisfying $d(i,j) \leq m$, $H_{i,j} = 0$.
Let $i,j\in [N]$ be a pair satisfying $d(i,j) = m+1$. Let $(i,k_1,\ldots,k_m,j)$ be a shortest path from $i$ to $j$ in $\mathcal G_{L^*}$ and let $J=\{k_0,k_1,\ldots,k_m,k_{m+1}\}$, where $k_0=i$ and $k_{m+1}=j$. Note that the graph $\mathcal G_{L_J^*}$ induced by $L_J^*$ is a path graph and that for all $s,t=0,\ldots,m+1$ satisfying $|s-t|\leq m$, $d(k_s,k_t)=|s-t|\leq m$, yielding $H_{k_s,k_t}=0$ by induction. Hence,
\begin{equation} \label{Nolabel1354}
	\Tr\left((L^*_J)^{-1} H_J\right) = 2\left((L^*_J)^{-1}\right)_{i,j}H_{i,j}=0,
\end{equation}
by \eqref{ConstTrace} with $J=\{i,j\}$. Let us show that $\left((L^*_J)^{-1}\right)_{i,j}\neq 0$, which will imply that $H_{i,j}=0$.
By writing $(L^*_J)^{-1}$ as the ratio between the adjugate of $L_J^*$ and its determinant, we have
\begin{equation} \label{ratioDets}
	\left((L^*_J)^{-1}\right)_{i,j}=\frac{\det L_{J\setminus\{i\},J\setminus\{j\}}}{\det L_J},
\end{equation}
where $L_{J\setminus\{i\},J\setminus\{j\}}$ is the submatrix of $L_J$ obtained by deleting the $i$-th line and $j$-th column. The determinant of this matrix can be expanded as
\begin{align}
\label{FormDet}	\det L_{J\setminus\{i\},J\setminus\{j\}} & = \sum_{\sigma\in\mathcal M_{i,j}}\varepsilon(\sigma)L_{i,\sigma(i)}^*L_{k_1,\sigma(k_1)}^*\ldots L_{k_m,\sigma(k_m)}^*\,,
	\end{align}
where $\mathcal M_{i,j}$ stands for the collection of all one-to-one maps from $J\setminus\{j\}$ to $J\setminus\{i\}$ and, for any such map $\sigma$, $\varepsilon(\sigma)\in\{-1,1\}$. There is only one term in \eqref{FormDet} that is nonzero: Let $\sigma\in\mathcal M_{i,j}$ for which the product in \eqref{FormDet} is nonzero. Recall that the graph induced by $L_J^*$ is a path graph. Since $\sigma(i)\in J\setminus\{i\}$, $L_{i,\sigma(i)}^*=0$ unless $\sigma(i)=k_1$. Then, $L_{k_1,\sigma(k_1)}^*$ is nonzero unless $\sigma(k_1)=k_1$ or $k_2$. Since we already have $\sigma(i)=k_1$ and $\sigma$ is one-to-one, $\sigma(k_1)=k_2$. By induction, we show that $\sigma(k_s)=k_{s+1}$, for $s=1,\ldots,m-1$ and $\sigma(k_m)=j$. As a consequence, $\det L_{J\setminus\{i\},J\setminus\{j\}}^*\neq 0$ and, by \eqref{Nolabel1354} and \eqref{ratioDets}, $H_{i,j}=0$, which we wanted to prove. 

Hence, by induction, we have shown that if $\diff^2\Phi(L^*)(H,H)=0$, then for any pair $i,j\in[N]$ such that $d(i,j)$ is finite, i.e., with $i\sim_{L^*} j$, $H_{i,j}=0$. \vspace{3mm} 

Let us now prove the converse statement: Let $H\in\SN$ satisfy $H_{i,j}=0$, for all $i,j$ with $i\sim_{L^*} j$. First, using Lemma \ref{SpanNullSpace} with its notation, for any $J\subseteq [N]$ and $j=1,\ldots,k$,
\begin{align*}
	D_J^{(j)}(L_J^*)^{-1}D_J^{(j)} & = \left(D_J^{(j)}L_J^*D_J^{(j)}\right)^{-1}  = (L_J^*)^{-1}
\end{align*}
and 
\begin{align*}
	D_J^{(j)}H_J^{(j)}D_J^{(j)}  = -H_J^{(j)}\,.
\end{align*}

Hence,
\begin{align*}
	\Tr\left((L_J^*)^{-1}H^{(j)}_J\right) & = \Tr\left(D^{(j)}(L_J^*)^{-1}D^{(j)}H^{(j)}_J\right) = - \Tr\left((L_J^*)^{-1}H_J^{(j)}\right)=0\,.
\end{align*}
Summing over $j=1,\ldots,k$ yields 
\begin{equation} \label{NullTraceJ}
	\Tr\left((L_J^*)^{-1}H_J\right)=0.
\end{equation}
In a similar fashion,
\begin{equation} \label{NullTrace}
	\Tr\left((I+L)^{-1}H\right)=0.
\end{equation}

Hence, using \eqref{keyEq2},
\begin{align*}
	\diff^2\Phi(L^*)(H,H) & = -\sum_{J\subseteq [N]}p_J^*\Tr^2\left((L_J^*)^{-1}H_J\right)+\Tr^2\left((I+L^*)^{-1}H\right)  = 0,
\end{align*}
which ends the proof of the theorem. \hfill \textsquare

\subsection{Proof of Proposition \ref{PropositionPath}}

Consider the matrix $H \in \SN$ with zeros everywhere but in positions $(1,N)$ and $(N,1)$, where its entries are $1$. Note that $\displaystyle{\Tr\left((L_J^*)^{-1}H_J\right)}$ is zero for all $J\subseteq [N]$ such that $J\neq [N]$. This is trivial if $J$ does not contain both $1$ and $N$, since $H_J$ will be the zero matrix. If $J$ contains both $1$ and $N$ but does not contain the whole path that connects them in $\mathcal G_{L^*}$, i.e., if $J$ does not contain the whole space $[N]$, then the subgraph $\mathcal G_{L_J^*}$ has at least two connected components, one containing $1$ and another containing $N$. Hence, $L_J^*$ is block diagonal, with $1$ and $N$ being in different blocks. Therefore, so is $(L_J^*)^{-1}$ and $\displaystyle{\Tr\left((L_J^*)^{-1}H_J\right)=2\left((L_J^*)^{-1}\right)_{1,N}=0}$.

Now, let $J=[N]$. Then, 
\begin{align}
	\Tr\left((L_J^*)^{-1}H_J\right) & = 2\left((L^*)^{-1}\right)_{1,N} \nonumber \\
	& = 2(-1)^{N+1}\frac{\det(L_{[N]\setminus\{1\},[N]\setminus\{N-1\}}^*)}{\det L^*} \nonumber \\
	\label{Path1} & = 2(-1)^{N+1}\frac{b^{N-1}}{\det L^*}.
\end{align}
Write $\det L^*=u_N$ and observe that 
\begin{equation*}
	u_k=au_{k-1}+b^2u_{k-2}, \hspace{3mm} \forall k\geq 2
\end{equation*}
and $u_1=a, u_2=a^2-b^2$. Since $a^2>4b^2$, there exists $\mu>0$ such that 
\begin{equation} \label{Toplitz}
	u_k\geq \mu\left(\frac{a+\sqrt{a^2-4b^2}}{2}\right)^k, \hspace{3mm} \forall k\geq 1.
\end{equation}
Hence, \eqref{Path1} yields
\begin{equation*}
	\left|\Tr\left((L_J^*)^{-1}H_J\right)\right| \leq \frac{2}{\mu|b|}\left(\frac{2|b|}{a+\sqrt{a^2-4b^2}}\right)^N,
\end{equation*}
which proves the second part of Proposition \ref{PropositionPath}, since $a+\sqrt{a^2-4b^2}>a>2|b|$. 

Finally note that \eqref{Toplitz} implies that all the principal minors of $L^*$ are positive so that $L \in \SSN$. \hfill \textsquare

\subsection{Proof of Theorem \ref{FourthOrder}}

Let $H\in\mathcal N(L^*)$. 
By Lemma \ref{Derivatives}, the third derivative of $\Phi$ at $L^*$ is given by
\begin{equation*}
	\diff^3\Phi(L^*)(H,H,H)=2\sum_{J\subseteq [N]}p_J^*\Tr\left(((L_J^*)^{-1}H_J)^3\right)-2\Tr\left(((I+L^*)^{-1}H)^3\right).
\end{equation*}
Together with \eqref{keyEq3}, it yields
\begin{align*} 
	\diff^3\Phi(L^*)(H,H,H) & = -\frac{2}{3}\Big(\sum_{J\subseteq [N]}p_J(L)a_{J,1}^3-a_1^3\Big) +\frac{4}{3}\Big(\sum_{J\subseteq [N]}p_J(L)a_{J,2}-a_2\Big) \\
	& \hspace{15mm} +\frac{2}{3}\Big(\sum_{J\subseteq [N]}p_J(L)a_{J,1}a_{J,2}-a_1a_2\Big).
\end{align*}
Each of the three terms on the right hand side of the above display vanish because of \eqref{NullTraceJ},  $H\in\mathcal N(L^*)$ and \eqref{NullTrace} respectively. This concludes the proof of~{\it (i)}.

Next, the fourth derivative of $\Phi$ at $L^*$ is given by 
\begin{equation*}
	\diff^4\Phi(L^*)(H,H,H,  H)=-6\sum_{J\subseteq [N]}p_J^*\Tr\big(((L_J^*)^{-1}H_J)^4\big)+6\Tr\big(((I+L^*)^{-1}H)^4\big).
\end{equation*}

Using \eqref{keyEq4} together with \eqref{NullTraceJ}, \eqref{NullTrace} and $\diff^3\Phi(L^*)(H,H,H)=0$, it yields \begin{equation*}
	\diff^4\Phi(L^*)(H,H,H,  H)=-\frac{2}{3}\Big(\sum_{J\subseteq [N]}p_J^*\Tr^2\big((L_J^*)^{-1}H_J)^2\big)-\Tr^2\big(((I+L^*)^{-1}H)^2\big)\Big).
\end{equation*}
Since $H\in\mathcal N(L^*)$, meaning $\diff^2\Phi(L^*)(H,H)=0$, we also have
\begin{equation*}
	\Tr\left(((I+L^*)^{-1}H)^2\right)=\sum_{J\subseteq [N]}p_J^*\Tr^2\left((L_J^*)^{-1}H_J)^2\right).
\end{equation*}
Hence, we can rewrite $\diff^4\Phi(L^*)(H,H,H,  H)$ as
\begin{equation*}
	\diff^4\Phi(L^*)(H,H,H,  H)=-\frac{2}{3}\big(\ELstar\left[\Tr^2\left((L_Z^*)^{-1}H_Z)^2\right)\right]-\ELstar\left[\Tr^2\left((L_Z^*)^{-1}H_Z)^2\right)\right]^2\big)\,.
\end{equation*}
This concludes the proof of~{\it (ii)}.

To prove~{\it (iii)}, note first that if $H=0$ then trivially $\diff^4\Phi(L^*)(H,H,H,  H)=0$. Assume now that $\diff^4\Phi(L^*)(H,H,H,  H)=0$, which, in view of~{\it (ii)} is equivalent to ${\Var[\Tr(((L_Z^*)^{-1}H_Z)^2)]= 0}$. Since $\Tr(((L_\emptyset^*)^{-1}H_\emptyset)^2)=0$, and $p_J^*>0$ for all $J \subset [N]$,  it yields 
\begin{equation}
\label{FourthOrderConst}
\Tr(((L_J^*)^{-1}H_J)^2)=0\quad \forall \, J \subset [N]\,.
\end{equation}

Fix $i,j\in[N]$. If $i$ and $j$ are in one and the same block of $L^*$, we know by Theorem \ref{MainCor} that $H_{i,j}=0$. On the other hand, suppose that $i$ and $j$ are in different blocks of $L^*$ and let $J=\{i,j\}$. Denote by $h=H_{i,j}=H_{j,i}$. Since $L_J^*$ is a $2\times 2$ diagonal matrix with nonzero diagonal entries and $H_{i,i}=H_{j,j}=0$, \eqref{FourthOrderConst} readily yields $h=0$. Hence, $H=0$, which completes the proof of~{\it (iii)}. \hfill \textsquare

\subsection{Proof of Theorem \ref{ThmCriticalPoints}}

Denote by $\Phi=\Phi_{L^*}$ and $K^*=L^*(I+L^*)^{-1}$. Let $L$ be the kernel of a partial decoupling of $Z$ according to a partition $\cP$ of $[N]$. By definition, the correlation kernel $K=L(I+L)^{-1}$ is block diagonal, with blocks $K_{J}=D_JK_{J}^*D_J, J \in \cP$, for some matrix $D \in \cD$. Without loss of generality, assume that $D=I$.  Since $L=K(I-K)^{-1}$, $L$ is also block diagonal, with blocks $L_{J}=K_{J}^*(I_{J}-K_{J}^*)^{-1}, J \in \cP$. To see that $L$ is a critical point of $\Phi$, note that the first derivative of $\Phi$ can be written in matrix form as
\begin{equation} \label{DerivMat}
	\diff\Phi(L)=\sum_{J'\subseteq [N]} p_{J'}^* L_{J'}^{-1} - (I+L)^{-1},
\end{equation}
where $L_{J'}^{-1}$ stands for the $N\times N$ matrix with the inverse of $L_{J'}$ on block $J'$ and zeros everywhere else.
Note that since $L$ is block diagonal, so are each of the terms of the right-hand side of \eqref{DerivMat}, with the same blocks. Hence, it is enough to prove that for all $J \in \cP$, the block $J$ of $\diff\Phi(L)$ (i.e., $\displaystyle{\left(\diff\Phi(L)\right)_{J}}$) is zero. Using elementary block matrix operations, for all $J\subseteq [N]$, the block $J$ of $L_{J'}^{-1}$ is given by $L_{J\cap J'}^{-1}$, using the same abuse of notation as before. Hence, the block $J$ of $\diff\Phi(L)$ is given by
\begin{equation*} 
	\left(\diff\Phi(L)\right)_{J}=\sum_{J'\subseteq [N]} p_{J'}^* L_{J'\cap J}^{-1} - (I_{J}+L_{J})^{-1},
\end{equation*}
which can also be written as
\begin{equation} \label{DerivMatBlock}
	\left(\diff\Phi(L)\right)_{J}=\sum_{J'\subseteq J} \tilde p_{J'}^* L_{J'}^{-1} - (I_{J}+L_{J})^{-1},
\end{equation}
where 
\begin{align}
	\tilde p_{J'}^* & = \sum_{J''\subseteq \widebar{J}}p_{J'\cup J''}^*  = \sum_{J''\subseteq \widebar {J}} \PLstar\left[Z=J'\cup J''\right]
	\label{RestrictPMF} = \PLstar\left[Z\cap J=J'\right].
\end{align}

Recall that $Z\cap J$ is a DPP on $J$ with correlation kernel $K_{J}^*$. Hence, its kernel is $L_{J}$ and \eqref{RestrictPMF} yields
\begin{equation*}
	\tilde p_{J'}^*=p_{J'}(L_{J}).
\end{equation*}
Together with \eqref{DerivMatBlock}, it yields
\begin{equation*}
	\left(\diff\Phi(L)\right)_{J}=\diff\Phi_{L_{J}}(L_{J}),
\end{equation*}
which is zero by Theorem \ref{MainThm}. This proves that $L$ is a critical point of $\Phi$.

Next, we prove that if $L$ is the kernel of a strict partial decoupling of $Z$, then it is a saddle point of $\Phi$. To that end, we exhibit two matrices $H, H' \in \SN$ such that $\diff^2\Phi(L)(H,H)>0$ and a $\diff^2\Phi(L)(H',H')<0$.

Consider a strict partial decoupling of $Z$ according to a partition $\cP$. Let $L$ and $K$ be its kernel and correlation kernel, respectively. In particular, there exists $J \in \cP$, $i\in J$ and $j\in\widebar J$ such that $K_{i,j}^*\neq 0$. Consider the matrix $H$ with zeros everywhere but in positions $(i,j)$ and $(j,i)$, where its entries are $1$. By simple matrix algebra,
\begin{align}
	& \hspace{10mm} \diff^2\Phi(L)(H,H) \nonumber \\
	& \hspace{15mm} = -\sum_{{J'}\subseteq [N]}p_{J'}^*\Tr\left((L_{J'}^{-1}H_{J'})^2\right) + \Tr\left(((I+L)^{-1}H)^2\right) \nonumber \\
	\label{saddle1} & \hspace{15mm} = -2\sum_{{J'}\subseteq [N]}p_{J'}^*\left(L_{{J'}\cap J}^{-1}\right)_{i,i}\left(L_{{J'}\cap \widebar J}^{-1}\right)_{j,j} + 2\left((I+L)^{-1}\right)_{i,i}\left((I+L)^{-1}\right)_{j,j},
\end{align}
where we recall that for all ${J'}\subseteq [N]$ and $k\in [N]$, $(L_{J'}^{-1})_{k,k}$ is set to zero if $k\notin {J'}$.

Denote by ${Y_i=(L_{Z\cap J}^{-1})_{i,i}}$ and ${Y_j=(L_{Z\cap \bar J}^{-1})_{j,j}}$. Note that ${\ELstar[Y_i]=\left((I+L)^{-1}\right)_{i,i}}$. Indeed, 
\begin{align*}
	\ELstar[Y_i] & = \sum_{{J'}\subseteq [N]}p_{J'}^*\left(L_{{J'}\cap J}^{-1}\right)_{i,i}  = \sum_{{J'}\subseteq J}\sum_{{J''}\subseteq \widebar J}p_{{J'}\cup {J''}}^*\left(L_{J'}^{-1}\right)_{i,i} \\
	& = \sum_{{J'}\subseteq J}\PLstar[Z\cap J={J'}](L_{J'}^{-1})_{i,i} = \sum_{{J'}\subseteq J}p_{J'}(L_J)\left(L_{J'}^{-1}\right)_{i,i}\\
	&  = (I_J+L_J)^{-1}_{i,i}  = (I+L)^{-1}_{i,i}.
\end{align*}
Here, the third equality follows from the fact that $L_J$ is the kernel of the DPP $Z\cap J$, the fourth equality follows from~\eqref{keyEq1MatrixForm} and the last equality comes from the block diagonal structure of $L$. It can be checked using the same argument that ${\ELstar[Y_j]=\left((I+L)^{-1}\right)_{j,j}}$. Together with \eqref{saddle1}, it yields
\begin{equation} \label{saddle2}
	\diff^2\Phi(L)(H,H)=-2\ELstar[Y_iY_j]+2\ELstar[Y_i]\ELstar[Y_j].
\end{equation}

Next, recall that $(X_1, \ldots, X_N)=\chi(Z)$ denotes the characteristic vector of $Z$ and observe that $Y_iY_j=0$ whenever $X_j=0$ or $X_j=0$ so that $Y_iY_j=Y_iY_jX_iX_j$. Hence,
%
$$
\ELstar[Y_iY_j]= \ELstar[Y_iY_j|X_i=1,X_j=1]\PLstar[X_i=1,X_j=1]\,.
$$
Since $L \in \SSN$, we have $\ELstar[Y_iY_j]>0$, yielding $\ELstar[Y_iY_j|X_i=1,X_j=1]>0$ by the previous equality. Moreover,
$$
\PLstar[X_i=1,X_j=1]=K_{i,i}^*K_{j,j}^*-(K_{i,j}^*)^2<K_{i,i}^*K_{j,j}^*=\PLstar[X_i=1]\PLstar[X_j=1]\,,
$$
where the inequality follows from the assumption $K_{i,j}^*\neq 0$. Hence,
\begin{equation}\label{saddle4}
\ELstar[Y_iY_j]< \ELstar[Y_iY_j|X_i=1,X_j=1]\PLstar[X_i=1]\PLstar[X_j=1]\,.
\end{equation}

We now use conditional negative association. To that end, we check that $Y_i=f_i(\chi(Z\cap J))$ and $Y_j=f_j(\chi(Z \cap \widebar{J}))$, for some non decreasing functions $f_i$ and $f_j$. For any $J' \subset J$, define $f_i(J')=(L_{J'}^{-1})_{i,i}$. It is sufficient to check that
\begin{equation} \label{saddle5}
	(L_{J'}^{-1})_{i,i} \leq (L_{J'\cup\{k\}}^{-1})_{i,i}\,, \quad \forall\, k \in J\setminus J'
\end{equation}
First, note that \eqref{saddle5} is true if $i\notin J'$, since in this case, $(L_{J'}^{-1})_{i,i}=0$ and $(L_{J'\cup\{k\}}^{-1})_{i,i}\ge 0$. Assume now that $i\in J'$ and consider the matrix $L_{J'\cup\{k\}}$, of which $L_{J'}$ is a submatrix. Using the Schur complement, we get that
\begin{equation}
\label{EQ:schur}
\big(L_{J'\cup\{k\}}^{-1}\big)_{J'}=\big(L_{J'}-\frac{1}{L_{k,k}}AA^\top\big)^{-1}\,,
\end{equation}
where $A=L_{J',\{k\}}$. Since $L_{k,k}>0$ and $AA^\top$ is positive semidefinite, then 
$$
L_{J'}-\frac{1}{L_{k,k}}AA^\top\preceq L_{J'}\,,
$$
where $\preceq$ denotes the L\"owner order on $\SSNs$. Moreover, it follows from the L\"owner-Heinz theorem that if $A \preceq B$, then $B^{-1} \preceq A^{-1}$ for any nonsingular $A,B \in \SN$. Therefore,
\begin{equation*}
	L_{J'}^{-1}\preceq \big(L_{J'}-\frac{1}{L_{k,k}}AA^{\top}\big)^{-1}\,.
\end{equation*}
In particular, the above display yields, together with~\eqref{EQ:schur},
\begin{equation*}
	\big(L_{J'}^{-1}\big)_{i,i}\preceq \big(\big(L_{J'}-\frac{1}{L_{k,k}}AA^{\top}\big)^{-1}\big)_{i,i}=\big(L_{J'\cup\{k\}}^{-1}\big)_{(i,i)}\,.
\end{equation*}
This completes the proof of~\eqref{saddle5} and monotonicity of $f_j$ follows from the same arguments. 

We are now in a position to use the conditional negative association property from Lemma~\ref{LEM:CNA}. Together with \eqref{saddle4}, it yields
\begin{equation}
\label{saddle6}
\ELstar[Y_iY_j] < \ELstar[Y_i|X_i=1,X_j=1]\ELstar[Y_2|X_i=1,X_j=1]\PLstar[X_i=1]\PLstar[X_j=1]\,. 
\end{equation}
Next, note that 
$$
\ELstar[Y_i|X_i=1,X_j=1]\leq \ELstar[Y_i|X_i=1]\,,
$$
and
$$
\ELstar[Y_j|X_i=1,X_j=1]\leq \ELstar[Y_j|X_j=1]\,.
$$
These inequalities are also a consequence of the conditional negative association property. Indeed, using Bayes formula and the fact that $j \notin J$ respectively, we get
\begin{align*}
\ELstar[Y_i|X_i=1,X_j=1]& = \frac{\ELstar[Y_iX_j|X_i=1]}{\ELstar[X_j|X_i=1]} \\
	& \leq \frac{\ELstar[Y_i|X_i=1]\ELstar[X_j|X_i=1]}{\ELstar[X_j|X_i=1]} = \ELstar[Y_i|X_i=1]\,.
\end{align*}
The second inequality follows from the same argument and the fact that $i \notin \widebar{J}$.
Finally, \eqref{saddle6} becomes
\begin{equation*}
	\ELstar[Y_iY_j] < \ELstar[Y_i]\ELstar[Y_j]
\end{equation*}
and hence, \eqref{saddle2} yields that $\diff^2\Phi(L)(H,H)>0$.

We now exhibit $H'$ such that  $\diff^2\Phi(L)(H,H)<0$. To that end, let $H'$ be the matrix with zeros everywhere but in position $(1,1)$, where $H'_{1,1}=1$. Let $J$ be the element of $\cP$ that contains $1$. By simple matrix algebra,
\begin{align}
	\diff^2\Phi(L)(H',H') & = -\sum_{J'\subseteq [N]}p_{J'}^*\big(L_{J'}^{-1}\big)_{1,1}^2 + \big((I+L)^{-1}\big)_{i,i}^2 \nonumber \\
	& = -\sum_{J'\subseteq J}\sum_{J''\subseteq \widebar J}p_{J'\cup J''}^*\big(L_{J'}^{-1}\big)_{1,1}^2 + \big((I+L)^{-1}\big)_{i,i}^2 \nonumber \\
	& = -\sum_{J'\subseteq J}\Big(\sum_{J''\subseteq \widebar J}p_{J'\cup J''}^*\Big)\big(L_{J'}^{-1}\big)_{1,1}^2 + \big((I_J+L_J)^{-1}\big)_{i,i}^2 \nonumber \\
	& = -\sum_{J''\subseteq J}p_{J'}(L_J)\big(L_{J'}^{-1}\big)_{1,1}^2 + \big((I+L)^{-1}\big)_{i,i}^2 \nonumber \\
	\label{saddle7} & = \diff^2\Phi_{L_J}(H'_J,H'_J).
\end{align}
By Theorem \ref{MainThm}, $\diff^2\Phi_{L_J}(H'_J,H'_J)\leq 0$. In addition, by Theorem \ref{MainCor}, $\diff^2\Phi_{L_J}(H'_J,H'_J)\neq 0$ since $H'_J$ has at least one nonzero diagonal entry. Hence, $\diff^2\Phi_{L_J}(H'_J,H'_J)< 0$ and it follows from \eqref{saddle7} that $\diff^2\Phi(L)(H',H')<0$, which completes the proof of Theorem \ref{ThmCriticalPoints}. \hfill \textsquare

\subsection{Proof of Proposition \ref{PropConjecture}}

Let $L$ be a critical point of $\Phi$ and $K=L(I+L)^{-1}$. Then, for all $N\times N$ matrices~$H$,
$$
	\diff\Phi(L)(H)=\sum_{J\subseteq [N]}p_J^*\Tr\left(L_J^{-1} H_J\right)-\Tr\left((I+L)^{-1}H\right)=0.
$$
Fix $t_1,\ldots,t_N \in \R$ and define $T=\Diag(t_1,\ldots,t_N)$, $H=LT$. Then, since $T$ is diagonal, $H_J=L_JT_J$, for all $J\subseteq [N]$. Using the above equation and the fact that $L$ and $(I+L)^{-1}$ commute, we have
\begin{equation} \label{CriticalPoints2}
	\sum_{J\subseteq [N]}p_J^*\sum_{j\in J} t_j=\Tr(KT)=\sum_{j=1}^N K_{j,j}t_j.
\end{equation}
Since \eqref{CriticalPoints2} holds for any $t_1,\ldots,t_N \in \R$, we conclude that 
\begin{equation*}
	K_{j,j}=\sum_{J\subseteq [N]:J\ni j}p_J^*=K_{j,j}^*,
\end{equation*}
for all $j\in [N]$, which ends the proof. \hfill \textsquare

\subsection{Proof of Theorem \ref{Consistency}}

Our proof is based on Theorem 5.14 in \cite{Vaa98}. We need to prove that there exists a compact subset $E$ of $\SSN$ such that $\hat L\in E$ eventually almost surely. 
Fix $\alpha,\beta \in (0,1)$ to be chosen later such that  $\alpha<\beta$ and define the compact set  of $\SSN$ as
$$
E_{\alpha,\beta}=\big\{L \in \SSN\,:\, K=L(I+L)^{-1} \in \cS_{[N]}^{[\alpha, \beta]}\big\}\,.
$$

Let $\delta=\min_{J\subseteq [N]} p_J^*$. Since $L^*$ is definite positive, $\delta>0$. Define the event $\mathcal A$ by
$$
\cA=\bigcap_{J \subset [N]}\big\{p_J^*\leq 2\hat p_J\leq 3p_J^*\big\}\,.
$$
and observe that on $\mathcal A$, we have $3\Phi(L)\leq 2\hat\Phi(L)\leq \Phi(L)$ simultaneously for all $L\in\SSN$.
In particular, 
\begin{equation} \label{ChainIneq}
	\Phi(\hat L) \geq 2\hat\Phi(\hat L)\geq 2\hat\Phi(L^*)\geq 3\Phi(L^*),
\end{equation}
where the second inequality follows from the definition of the MLE. 

Using Hoeffding's inequality together with a union bound, we get
\begin{equation} \label{ProbaEventA}
	\PLstar[\mathcal A]\geq 1-2^{N+1}e^{-\delta^2 n/2}\,.
\end{equation}
Observe that $\Phi(L^*)<0$, so we can define $\alpha<\exp(3\Phi(L^*)/\delta)$ and $\beta>1-\exp(3\Phi(L^*)/\delta)$ such that $0<\alpha<\beta<1$. Let  $L\in\SSN\setminus E_{\alpha,\beta}$ and $K=L(I+L)^{-1}$. Then, either {\it (i)} $K$ has an eigenvalue that is less than $\alpha$, or {\it (ii)} $K$ has an eigenvalue that is larger than $\beta$. Since all the eigenvalues of $K$ lie in $(0,1)$, we have that  $\det(K)\leq \alpha$ in case {\it (i)} and $\det(I-K)\leq  1-\beta$ in case  {\it (ii)}. Recall that 
\begin{align*}
	\Phi(L) = \sum_{J\subseteq [N]} p_J^*\log |\det(K-I_{\bar J})|,,
\end{align*}
and observe that each term in this sum is negative. Hence, by definition of $\alpha$ and $\beta$,
$$
	\Phi(L) \leq 
	\left\{
	\begin{array}{ll}
	 p_{[N]}^*\log\alpha \leq \delta \log\alpha  < 3\Phi(L^*)\leq \Phi(\hat L)\, & \text{in case {\it (i)}}\\
	 p_\emptyset^*\log(1-\beta)  \leq \delta \log(1-\beta)  < 3\Phi(L^*)\leq \Phi(\hat L)\, & \text{in case {\it (ii)}}
	 \end{array}
\right.
$$
using \eqref{ChainIneq}. Thus, on $\cA$, $\Phi(L)<\Phi(\hat L)$ for all $L \in\SSN\setminus E_{\alpha,\beta}$. It yields that on this event, $\hat L\in E_{\alpha,\beta}$.

Now, let $\varepsilon>0$. For all $J\subseteq [N]$, $p_J(\cdot)$ is a continuous function; hence, we can apply Theorem 5.14 in \cite{Vaa98}, with the compact set $E_{\alpha,\beta}$. This yields
\begin{align*}
	\PLstar[\ell(\hat L,L^*)>\varepsilon] & \leq \PLstar[\ell(\hat L,L^*)>\varepsilon, \hat L\in E_{\alpha,\beta}]+\PLstar[\hat L\notin E_{\alpha,\beta}] \\
	& \leq \PLstar[\ell(\hat L,L^*)>\varepsilon, \hat L\in E_{\alpha,\beta}]+\left(1-\PLstar[\mathcal A]\right).
\end{align*}
Using Theorem 5.14 in \cite{Vaa98}, the first term goes to zero, and the second term goes to zero by \eqref{ProbaEventA}. This ends the proof of Theorem \ref{Consistency}. \hfill \textsquare

\subsection{Proof of Theorem \ref{AsympNormMLE}}

The first statement of Theorem \ref{AsympNormMLE} follows from Theorem 5.52 in \cite{Vaa98}, with $\alpha=4$ and $\beta=1$ (the fact that $\beta=1$ being a consequence of the proof of Corollary 5.53 in \cite{Vaa98}). For the second statement, note that since the DPPs $Z\cap J, J \in \cP$ are independent, each $\hat L_{J}, J \in \cP$ is the maximum likelihood estimator of $L_{J}^*$. Since $L_{J}^*$ is irreducible, the $n^{1/2}$-consistency of $\hat L_{J}$ follows from Theorem \ref{AsymNormMLE}. \hfill \textsquare

\bibliographystyle{alphaabbr}
\bibliography{Biblio}

\end{document}